\documentclass[a4paper,
               11pt,
               pdftex,
               normalheadings,
               headsepline,
               footsepline,
               onecolumn,
               headinclude,
               footinclude,
               DIV14,
               abstracton
               ]{scrartcl}

\usepackage
{
    graphicx,
    amssymb,
    amsmath,
    amsthm,
    xcolor,
    dsfont,
    algpseudocode,
    authblk,
}

\usepackage[bf]{caption}
\usepackage[colorlinks,
            pdffitwindow=false,
            plainpages=false,
            pdfpagelabels=true,
            pdfpagemode=UseOutlines,
            pdfpagelayout=SinglePage,
            bookmarks=false,
            colorlinks=true,
            hyperfootnotes=false,
            linkcolor=blue,
            citecolor=green!50!black]{hyperref}

\DeclareMathAlphabet{\mathpzc}{OT1}{pzc}{m}{it}

\newcommand{\subfiguretitle}[1]{{\scriptsize{#1}} \\[1mm] }
\newcommand{\R}{\mathbb{R}}
\newcommand{\C}{\mathbb{C}}

\providecommand{\abs}[1]{\left\lvert #1 \right\rvert}

\newcommand\xqed[1]{\leavevmode\unskip\penalty9999 \hbox{}\nobreak\hfill \quad\hbox{#1}}
\newcommand{\exampleSymbol}{\xqed{$\triangle$}}

\DeclareMathOperator{\trace}{trace}
\DeclareMathOperator{\rank}{rank}
\DeclareMathOperator{\Aut}{Aut}

\newtheorem{theorem}{Theorem}[section]
\newtheorem{corollary}[theorem]{Corollary}
\newtheorem{lemma}[theorem]{Lemma}
\newtheorem{proposition}[theorem]{Proposition}
\newtheorem{definition}[theorem]{Definition}
\newtheorem{remark}[theorem]{Remark}
\newtheorem{remarks}[theorem]{Remarks}
\newtheorem{example}[theorem]{Example}

\makeatletter
\renewcommand*\env@matrix[1][*\c@MaxMatrixCols c]{%
  \hskip -\arraycolsep
  \let\@ifnextchar\new@ifnextchar
  \array{#1}}
\makeatother

\usepackage{amsopn}

\begin{document}

\title{Self-adjoint Matrices are Equivariant}
\author[*]{Michael Dellnitz}
\affil[*]{\normalsize Department of Mathematics, Paderborn University, D-33095 Paderborn, Germany}

\maketitle

\begin{abstract}
In this short note we prove that a matrix $A\in\R^{n,n}$ is self-adjoint if and only if
it is equivariant with respect to the action of a group $\Gamma\subset {\bf O}(n)$
which is isomorphic to $\otimes_{k=1}^n\mathbf{Z}_2$. Moreover we discuss
potential applications of this result, and we use it in particular for the approximation of
higher order derivatives for smooth real valued functions of several variables.
\end{abstract}



\pagestyle{myheadings}
\thispagestyle{plain}


\section{Introduction}

Within this short note we prove a characterization for a matrix being \emph{symmetric} --
in the sense of $A = A^T$ -- by using the notion of \emph{equivariance}.
The proof of this fact is not difficult at all, but to the best of the knowledge
of the author so far the related result cannot explicitly be found in the literature.

However, in several articles concerning the development of dynamical
systems for the solution of certain optimization problems this underlying equivariance
structure is implicitly present (e.g.\ \cite{Scho68, Bro88, Bro89}), and one would expect that
this is also the case in other applications. The point
of this note is to state this characterization of $A = A^T$ explicitly, and this is done in Section~\ref{sec:mr}.
In Section~\ref{sec:impli} we discuss potential applications in equivariant
bifurcation theory, and we illustrate concretely how this result can be used for the
construction of simple 
approximations of derivatives of higher order for real valued functions.

\section{Main Result}\label{sec:mr}

Let $\Sigma \subset {\bf O}(n)$ be the abelian group consisting of the
$2^n$ matrices
\[
    \begin{pmatrix}
        \pm 1   & 0 & 0 & \cdots   & 0 \\
        0 & \pm 1   & 0 &\cdots & 0   \\
        \vdots  & \vdots & \vdots & \cdots  & \vdots \\
        0 & \cdots   & 0 & 0 & \pm 1 
    \end{pmatrix}.
\]
Obviously for any diagonal matrix 
\[
D = 
    \begin{pmatrix}
        \lambda_1   & 0 & 0 & \cdots   & 0 \\
        0 & \lambda_2   & 0 &\cdots & 0   \\
        \vdots  & \vdots & \vdots & \cdots  & \vdots \\
        0 & \cdots   & 0 & 0 & \lambda_n 
    \end{pmatrix},\quad \lambda_j \in \R, \quad j=1,2,\ldots,n,
\]
we have
\begin{equation*}
\sigma D = D \sigma \quad \forall \sigma \in \Sigma.
\end{equation*}
In fact, it is easy to verify that for an arbitrary matrix $B\in\R^{n,n}$ one has
\begin{equation}
\label{eq:DSigma}
\sigma B = B \sigma \quad \forall \sigma \in \Sigma \quad \Longleftrightarrow \quad
\mbox{$B$ is a diagonal matrix.}
\end{equation}

In this note we prove the following characterization:

\begin{proposition}
\label{prop:main}
A matrix $A\in \R^{n,n}$ is self-adjoint (i.e. $A = A^T$) if and only if there is
an orthogonal matrix $V\in {\bf O}(n)$ such that
\begin{equation}
\label{eq:equiv}
\gamma A = A \gamma\quad \forall \gamma \in \Gamma,
\end{equation}
where the group $\Gamma \subset {\bf O}(n)$ is defined by
\[
\Gamma = \{ V^T \sigma V : \sigma \in \Sigma \}.
\]
\end{proposition}

\begin{proof}
Suppose that $A = A^T$. Then there is
$V\in {\bf O}(n)$ such that
\[
D = V A V^T
\]
is a diagonal matrix.
By \eqref{eq:DSigma}
we have for all $\sigma \in \Sigma$
\begin{equation*}
\sigma VAV^T = VAV^T \sigma \quad \Longleftrightarrow
\quad V ^T\sigma VA = AV^T \sigma V.
\end{equation*}
Therefore $A$ satisfies the equivariance condition \eqref{eq:equiv}.

Now suppose that  \eqref{eq:equiv} is satisfied for some $V\in {\bf O}(n)$.
Then the matrix $V A V^T$ commutes with every $\sigma\in\Sigma$, and 
by \eqref{eq:DSigma}
it follows that $D=VAV^T$ is a diagonal matrix. Therefore
\[
A^T = (V^T D V)^T = A
\]
as desired.
\end{proof}

\begin{remarks}
\begin{itemize}
\item[(a)] Observe that the implication "$\Longrightarrow$" could also be proved by
using the well know fact that two matrices $A$ and $B$ commute if 
there is an orthogonal transformation $V$ such that both $V^T A V$ and
$V^T B V$ are diagonal.
\item[(b)] By construction all the eigenvalues of every $\gamma \in \Gamma$ are $1$ or $-1$.
In particular $\gamma ^2 = I$ for all $\gamma \in \Gamma$.
Moreover, by (a) the matrix $A$ and all $\gamma\in \Gamma$ possess the same set of eigenvectors.
\item[(c)] Obviously analogous results can be obtained for Hermitian or normal matrices:
Using essentially the same proof as in Proposition~\ref{prop:main} one can show that
a matrix $A\in \C^{n,n}$ is normal (i.e. $AA^* = A^* A$) if and only if there is
a unitary matrix $W\in {\bf U}(n)$ such that
\[
\gamma A = A \gamma\quad \forall \gamma \in \Gamma,
\]
where the group $\Gamma \subset {\bf U}(n)$ is defined by
\[
\Gamma = \{ W^* \sigma W : \sigma \in \Sigma \}.
\]
\end{itemize}
\end{remarks}

\section{On Applications}\label{sec:impli}

Proposition~\ref{prop:main} could be used to look at results for
symmetric matrices in the light of the equivariance condition \eqref{eq:equiv}.
For instance a result from \cite{GSS88} on the genericity of the structure
of eigenspaces would imply the well known fact that generically eigenspaces of self-adjoint
matrices are one-dimensional. (Simply observe that $\Gamma \cong \otimes_{k=1}^n\mathbf{Z}_2 $
possesses only one-dimensional (absolutely) irreducible representations.)

A potentially more interesting
application may be the analysis of symmetry breaking bifurcations
for gradient systems since in this case the Jacobian would be equivariant
according to \eqref{eq:equiv}. This could particularly be useful for
bifurcation problems where the (symmetric) steady state solution does not depend
on the bifurcation parameter. In fact, some time ago the author himself has
co-authored an article on "equivariant (and) self-adjoint matrices" \cite{DeMel94},
and it could be interesting to reconsider these results by taking 
the insight provided by Proposition~\ref{prop:main} into account.

However, within this note let us focus concretely on one
implication involving Taylor expansions. In this context the following
immediate consequence
of Proposition~\ref{prop:main} strongly indicates that the result
could, for instance, be used to develop a novel general approach for the
construction of higher order stencils for real valued functions
of several variables.

Suppose that $f:\R^n \to \R$ is smooth in a neighborhood of
$\bar x\in\R^n$. In the following we use Proposition~\ref{prop:main} to
construct a four-point-stencil which provides a second order approximation of
evaluations of the fourth order derivative in $\bar x$. For convenience
we write the Taylor expansion of $f$ in $\bar x$ as
\[
f(\bar x+h) = f(\bar x) + \nabla f(\bar x)^T h + \frac{1}{2} h^T H(\bar x) h + \sum_{j=3}^\infty g_j(\bar x,h),
\]
where $g_j(\bar x,h) = O(\| h\|^j)$, $j=3,4,\ldots$, and $H(\bar x)$ is the Hessian matrix of $f$ at $\bar x$.
\begin{corollary}
Denote by $\Gamma(\bar x)$ the group in Proposition~\ref{prop:main}
corresponding to the Hessian matrix $H(\bar x)$.
Then for all $\gamma \in\Gamma(\bar x)$ we have
\begin{equation}
\label{eq:Taylor1}
f(\bar x+\gamma h) -2f(\bar x)+f(\bar x-\gamma h) = h^T H(\bar x) h + 2 g_4(\bar x,\gamma h) + O(\| h\|^6),
\end{equation}
and therefore for all $\gamma_1, \gamma_2\in\Gamma(\bar x)$
\begin{equation}\label{eq:Taylor2}
\begin{array}{ll}
& f(\bar x+\gamma_1 h) + f(\bar x-\gamma_1 h) -  f(\bar x+\gamma_2 h) - f(\bar x-\gamma_2 h) = \\
= & 2(g_4(\bar x,\gamma_1 h) - g_4(\bar x,\gamma_2 h)) + O(\| h\|^6).
\end{array}
\end{equation}
In particular, $f(\bar x+\gamma_1 h) + f(\bar x-\gamma_1 h) -  f(\bar x+\gamma_2 h) - f(\bar x-\gamma_2 h) = O(\| h\|^4)$.
\end{corollary}

\begin{proof}
For $h\in\R^n$ and $\gamma_j \in\Gamma(\bar x)$ $(j=1,2)$ we compute
using \eqref{eq:equiv} and the fact that $\Gamma(\bar x) \subset {\bf O}(n)$
\[
f(\bar x\pm \gamma_j h) = 
f(\bar x) \pm \nabla f(\bar x)^T \gamma_j h + \frac{1}{2} h^T H(\bar x) h \pm g_3(\bar x,\gamma_j h) + g_4(\bar x,\gamma_j h) \pm g_5(\bar x,\gamma_j h) + \cdots
\]
Therefore
\begin{eqnarray*}
f(\bar x+\gamma_1 h) +f(\bar x-\gamma_1 h) & = & 2 \left( f(\bar x)  + \frac{1}{2} h^T H(\bar x) h + g_4(\bar x,\gamma_1 h) + O(\| h\|^6)\right) \\
f(\bar x+\gamma_2 h) +f(\bar x-\gamma_2 h) & = & 2 \left( f(\bar x)  + \frac{1}{2} h^T H(\bar x) h + g_4(\bar x,\gamma_2 h) + O(\| h\|^6)\right), 
\end{eqnarray*}
and \eqref{eq:Taylor1}, \eqref{eq:Taylor2} immediately follow.
\end{proof}

Obviously, if $\gamma_1 = \pm \gamma_2$ then
this result is not useful. However, for all other
choices of $\gamma_j$ this leads to interesting approximations of the fourth order derivative
as long as $h$ is not an eigenvector of $\gamma_j$ ($j=1,2$).

\begin{example}
Let $f:\R^3 \to \R$ be defined by
\[
f(x_1,x_2,x_3) = x_1 x_2 x_3^2 +x_1^2 - 3x_2^2 + x_2 \sin(x_1) - x_2^2 x_3^2.
\]
We choose $\bar x = (1,1,1)^T$ and compute
\[
H(\bar x) = \begin{pmatrix}
        2- \sin(1) & 1 + \cos(1) & 2 \\ 
        1+\cos(1) & -8 & -2 \\
        2 & -2 & 0 
    \end{pmatrix}.
\]
The choice of
\[
\sigma_1 = I\quad \mbox{and}\quad
\sigma_2 = 
    \begin{pmatrix}
        -1  & 0 & 0\\
        0  & 1 & 0 \\
        0 &  0 & 1 
    \end{pmatrix}
\]
leads to
\[
\gamma_1 = I\quad \mbox{and}\quad
\gamma_2 = 
    \begin{pmatrix}
        0.9225 & 0.3723 & 0.1015\\
        0.3723  & -0.7896 & -0.4877 \\
        0.1015 &  -0.4877 & 0.8671 
    \end{pmatrix}.
\]
For $h=(0.2,0.05,0.1)^T$ we obtain
\[
f(\bar x+h) + f(\bar x-h) -  f(\bar x+\gamma_2 h) - f(\bar x-\gamma_2 h) \approx 6.40\cdot 10^{-5},
\]
and for $h=\frac{1}{10}(0.2,0.05,0.1)^T$ one computes
\[
f(\bar x+h) + f(\bar x-h) -  f(\bar x+\gamma_2 h) - f(\bar x-\gamma_2 h) \approx 6.38\cdot 10^{-9}
\]
as expected.
\end{example}

\bibliographystyle{unsrt}
\bibliography{SAEQ}

\end{document}